\documentclass[12pt,reqno]{amsart}
\usepackage{amssymb,latexsym,graphicx,amscd,upgreek}
\usepackage{lscape}
\usepackage[all]{xy}
\usepackage{color}
\usepackage{epic,eepic}
\usepackage{xspace}
\usepackage{mathrsfs}
\usepackage{setspace}
\usepackage{cases}
\usepackage{bbm}
\usepackage{hyperref}
\hypersetup{
    colorlinks,
    linkcolor={blue},
    citecolor={blue},
    urlcolor={blue!80!black}
}

\newtheorem{Thm}{Theorem}[section]
\newtheorem{Lem}[Thm]{Lemma}
\newtheorem{Cor}[Thm]{Corollary}
\newtheorem{Prop}[Thm]{Proposition}

\theoremstyle{definition}

\newcommand{\Example}{\noindent{\bf Example}}

\newcommand{\Remarks}{\noindent{\bf Remarks}}

\newcommand{\Z}{\mathbb{Z}}

\newcommand{\N}{\mathbb{N}}

\newcommand{\df}{\colon}

\newcommand{\cI}{{\mathcal I}}

\newcommand{\cO}{{\mathcal O}}
\newcommand{\cP}{{\mathcal P}}

\newcommand{\cZ}{{\mathcal Z}}

\newcommand{\bd}{\mathbf{d}}

\newcommand{\rk}{\operatorname{rk}}

\newcommand{\md}{\operatorname{mod}}

\newcommand{\gldim}{\operatorname{gl.dim}}
\newcommand{\pdim}{\operatorname{proj.dim}}
\newcommand{\idim}{\operatorname{inj.dim}}

\newcommand{\dimv}{\underline{\dim}}

\newcommand{\soc}{\operatorname{soc}}

\newcommand{\tp}{\operatorname{top}}

\newcommand{\proj}{\operatorname{proj}}
\newcommand{\inj}{\operatorname{inj}}

\newcommand{\Hom}{\operatorname{Hom}}
\newcommand{\Ext}{\operatorname{Ext}}

\newcommand{\End}{\operatorname{End}}

\newcommand{\reg}{\operatorname{reg}}

\newcommand{\taureg}{\operatorname{\tau-reg}}

\newcommand{\Ima}{\operatorname{Im}}

\newcommand{\ov}{\overline}

\newcommand{\bil}[1]{\langle #1\rangle}

\newcommand{\op}{{\rm op}}

\newcommand{\irr}{\operatorname{Irr}}

\newcommand{\bsm}{\begin{smallmatrix}}
\newcommand{\esm}{\end{smallmatrix}}

\newcommand{\bbsm}{\left[\begin{smallmatrix}}
\newcommand{\besm}{\end{smallmatrix}\right]}

\newcommand{\bbm}{\begin{matrix}}
\newcommand{\ebm}{\end{matrix}}

\newcommand{\calZ}{\mathcal{Z}}

\begin{document}

\date{12.10.2026}

\title{$E$-symmetric algebras I}

\author{Grzegorz Bobi\'nski}
\address{Grzegorz Bobi\'nski\newline
Faculty of Mathematics and Computer Science\newline
Nicolaus Copernicus University\newline
ul. Chopina 12/18\newline
87-100 Toru\'n\newline
Poland}
\email{gregbob@mat.umk.pl}

\author{Jan Schr\"oer}
\address{Jan Schr\"oer\newline
Mathematisches Institut\newline
Universit\"at Bonn\newline
Endenicher Allee 60\newline
53115 Bonn\newline
Germany}
\email{schroer@math.uni-bonn.de}


\begin{abstract}
We study various symmetry conditions related to the Auslander-Reiten formula for finite-dimensional algebras.
In particular, we define a new class of algebras which we call $E$-symmetric.
A theorem by Derksen--Weyman--Zelevinksy implies that
all Jacobian algebras are $E$-symmetric. 
We prove the converse for the class of monomial algebras.
\end{abstract}

\maketitle

\setcounter{tocdepth}{1}
\numberwithin{equation}{section}
\tableofcontents

\parskip2mm


\section{Introduction and main results}\label{sec:1}


Throughout, $K$ denotes an algebraically closed field, and
$A$ is always a finite-dimensional associative $K$-algebra.
Without loss of generality we assume that $A = KQ/I$, where $KQ$ is the path algebra of a quiver $Q$ and $I$ is an admissible ideal of $KQ$.
Let $\md(A)$ be the category of finite-dimensional left $A$-modules.

\subsection{Auslander-Reiten formula and
$E$-invariants}
The \emph{Auslander-Reiten translation} $\tau_A$ gives rise 
to equivalences 
$$
\xymatrix{
\underline{\md}(A) \ar@/^1ex/[r]^\tau & \ar@/^1ex/[l]^{\tau^-} 
\overline{\md}(A)
}
$$
of stable module categories.
For a proof of the following theorem (and the missing definitions) ,
we refer to \cite[Section~IV.4]{ARS97} or
\cite[Section~III.6]{SY11}.

\begin{Thm}[{Auslander, Reiten}]\label{thm:ARformula}
For $M,N \in \md(A)$, we have
$$
\Ext_A^1(M,N) \cong D\overline{\Hom}_A(N,\tau_A(M))
\cong D\underline{\Hom}_A(\tau_A^{-}(N),M).
$$
\end{Thm}

For $M,N \in \md(A)$, we define
$$
\hom_A(M,N) := \dim \Hom_A(M,N).
$$

Let
\begin{align*}
E(M,N) &:= \hom_A(N,\tau_A(M)),
\\
E^-(M,N) &:= \hom_A(\tau_A^{-}(N),M).
\end{align*}
Furthermore, set
$$
E(M) := E(M,M) \text{\quad and \quad}
E^-(M) := E^-(M,M).
$$
These numbers are called \emph{$E$-invariants} and
play an important role in Derksen--Weyman--Zelevinsky's
\cite{DWZ08, DWZ10} representation theoretic approach to 
Fo\-min--Zelevinsky cluster algebras.

\subsection{$\tau$-regular modules}
Let $n := |Q_0|$ be the number of vertices of $Q$.
This is then also the number of isomorphism classes of simple $A$-modules.
For a dimension vector $\bd \in \N^n$, let
$\md(A,\bd)$ be the affine variety of $A$-modules
with dimension vector $\bd$.
The corresponding product $G_\bd$ of general linear groups
acts on $\md(A,\bd)$ by conjugation.
The orbits of this action correspond to the isomorphism classes of $A$-modules with dimension vector $\bd$.
For $M \in \md(A,\bd)$, we denote its orbit by $\cO_M$.
Let $\irr(A,\bd)$ be the set of irreducible components of 
$\md(A,\bd)$, and let
$$
\irr(A) := \bigcup_{\bd \in \N^n} \irr(A,\bd).
$$

For $M \in \md(A,\bd)$, let
$$
\md_M(A,\bd) :=
\bigcup_{\substack{\calZ \in \irr(A,\bd)\\M \in \calZ}} \calZ,
$$
and define
$$
c(M) := \dim \md_M(A,\bd) - \dim \cO_M.
$$
Then 
$$
c(M) \le E(M)
$$ 
by Voigt's Lemma and the Auslander-Reiten
formula.
We call $M$ \emph{$\tau$-regular} (resp. \emph{$\tau^-$-regular}) 
if
$$
c(M) = E(M) \qquad(\text{resp. $c(M) = E^-(M)$}).
$$
Let $\taureg(A)$ be the subcategory of all $\tau$-regular modules in $\md(A)$.
These modules generalize $\tau$-rigid modules (= the protagonists of $\tau$-tilting theory) and modules $M$ with
$\pdim(M) \le 1$.

Studying $\tau$-regular modules corresponds to studying modules which have projective presentations of maximal rank, see Section~\ref{subsec:maximal} for details.

For each dimension vector $\bd$,
the $\tau$-regular modules in $\md(A,\bd)$ form an open subset,
see \cite[Lemma~4.4(ii)]{BS25}.
The Zariski closure  of this open subset is
therefore a union of irreducible components of $\md(A,\bd)$.
The set of these components is denoted by $\irr^\tau(A,\bd)$.
The components in $\irr^\tau(A,\bd)$ are called 
\emph{generically $\tau$-regular} or just 
\emph{$\tau$-regular} for short.
This class of irreducible components was first introduced in \cite{GLS12}.

Dually, one defines $\tau^-$-$\reg(A)$  and
$\irr^{\tau^-}(A,\bd)$.

A module $M \in \md(A)$ is
\emph{$\tau$-rigid} 
if
$$
E(M) = 0.
$$
In this case, $M$ is $\tau$-regular, and we have $\overline{\cO_M} \in \irr^\tau(A)$.

\subsection{$1$-Iwanaga-Gorenstein algebras}
Let 
\begin{align*}
\cP_{\le 1}(A) &:= \{ M \in \md(A) \mid \pdim(M) \le 1 \},
\\
\cI_{\le 1}(A) &:= \{ M \in \md(A) \mid \idim(M) \le 1 \}.
\end{align*}
Analogously, one defines $\cP_{< \infty}(A)$ and
$\cI_{< \infty}(A)$.
An algebra $A$ is \emph{$1$-Iwanaga-Gorenstein}
if
$$
\pdim(D(A_A)) \le 1
\text{\quad and \quad} 
\idim({_A}A) \le 1.
$$
This is equivalent to the condition
$$
\cP_{\le 1}(A) = \cI_{\le 1}(A) = \cP_{< \infty}(A)
= \cI_{< \infty}(A).
$$

Trivial examples of $1$-Iwanaga-Gorenstein algebras are
selfinjective algebras and hereditary algebras.
But there are many more.

\subsection{Four symmetry conditions}
Our aim is to study the following symmetry conditions:
\begin{itemize}\itemsep2mm

\item[(S1)]
$E(M) = E^-(M)$ for all $M \in \md(A)$;

\item[(S2)]
$\taureg(A) = \tau^-$-$\reg(A)$;

\item[(S3)]
$\irr^\tau(A) = \irr^{\tau^-}(A)$;

\item[(S4)]
$A/AeA$ is $1$-Iwanaga-Gorenstein for each $e = e^2$ in
$A$.

\end{itemize}
We have
$$
\text{(S1)} \implies \text{(S2)} \implies \text{(S3)}
\implies \text{(S4)}.
$$
The first and second implications follow directly from the definitions, whereas 
the third implication is non-trivial and corresponds to
\cite[Theorem~1.7]{BS25}.

We call $A$ an \emph{$E$-symmetric algebra} if it satisfies \text{(S1)}.

\Remarks:
\begin{itemize}\itemsep2mm

\item[(i)]
Let $N \in \md(A)$ be indecomposable,
and let
$$
\tau^{\Z}(N) := \{ \tau_A^i(N) \mid i \in \Z \} \setminus \{ 0 \}.
$$
Condition (S1) implies that
$E(M)$
is constant for all $M \in \tau^{\Z}(N)$.

\item[(ii)]
At present it is an open problem if 
$\text{(S3)} \implies \text{(S2)}$ holds. (This question was first posed by Pfeifer \cite{Pf25}.)

\item[(iii)]
The implication $\text{(S2)} \implies \text{(S1)}$
does not hold in general:
Let $A = KQ/I$ where $Q$ is the quiver
$$
\xymatrix{
1 \ar@(ul,dl)[]_a \ar@(ur,dr)[]^b
}
$$
and $I$ is generated by $\{ a^2,\; b^2,\; ab-ba \}$.
This is a local selfinjective algebra, and we have
$$
\taureg(A)  = \proj(A) = \inj(A) = \text{$\tau^-$-$\reg(A)$}.
$$
(These equalities hold for all local selfinjective algebras.)
Let $P = {_A}A$, and let $M = P/\soc(P)$.
Then 
$E(M) = 3$ and $E^-(M) = 6$.
Thus $A$ is not $E$-symmetric.

\item[(iv)]
The implication $\text{(S4)} \implies \text{(S3)}$
does not hold in general:
Let $A = KQ/I$ where $Q$ is the quiver
$$
\xymatrix{
1 \ar@/^1ex/[r]^a & 2 \ar@/^1ex/[l]^b
}
$$
and $I$ is generated by $\{ ab,\; ba \}$.
Note that $A$ is selfinjective.
The Auslander-Reiten quiver of $A$ looks as follow:
$$
\xymatrix@-2ex{
& {\bbm 2\\1 \ebm} \ar[dr]
&& {\bbm 2\\1 \ebm} \ar[dr]
\\
{\bbm 1 \ebm}\ar[ur]
&&
{\bbm 2 \ebm}\ar[ur]\ar@{-->}[ll]
&&
{\bbm 1 \ebm}\ar@{-->}[ll]
}
$$
Clearly $A$ satisfies (S4).
Let 
$$
M = {\bbm 2\\1 \ebm} \oplus {\bbm 2 \ebm}.
$$
Then 
$E(M) = 0$ and $E^-(M) = 1$.
It follows that $\calZ := \ov{\cO_M} \in \irr^\tau(A)$ and
$\calZ \notin \irr^{\tau^-}(A)$.

\end{itemize}

\subsection{Jacobian algebras are $E$-symmetric}
Under certain finiteness conditions, the following theorem is due to Derksen--Weyman--Zelevinsky \cite[Proposition~7.3, Corollary~10.9]{DWZ10}.
A proof of the general case can be found in \cite{GHLS26}.
In particular, we do not need to assume that $Q$ is $2$-acyclic or that
${\rm char}(K) = 0$.

\begin{Thm}[\cite{DWZ10}, \cite{GHLS26}]\label{thm:jacobian}
Let $Q$ be a quiver, and let $\tilde{W}$ be a reduced hyperpotential
for $Q$.
Then the Jacobian algebra $\cP(Q,\tilde{W})$
satisfies 
$$
E(M) = E^-(M)
$$
for all $M \in \md(\cP(Q,\tilde{W}))$.
\end{Thm}

We are only interested in the case where $\cP(Q,\tilde{W})$ is finite-dimensional.
But Theorem~\ref{thm:jacobian} holds in general,
see \cite{GHLS26} for details.
The notion of a hyperpotential is due to Ladkani \cite{L14}.

Our aim is to show that for some classes of algebras being $E$-symmetric is equivalent to being a Jacobian algebra.


\subsection{Main results}
We call $A = KQ/I$ a \emph{monomial algebra} if the admissible ideal $I$ 
is generated by a set of paths (of length at least two) in $Q$.
These algebras form an important class of finite-dimensional algebras. They admit a combinatorial description of their projective and injective modules, which makes them accessible to homological computations.

\begin{Thm}\label{thm:main1}
Let $A = KQ/I$ be a monomial algebra.
Then the following are equivalent:
\begin{itemize}\itemsep2mm

\item[(i)]
$E(M) = E^-(M)$ for all $M \in \md(A)$ {\rm (= (S1))};

\item[(ii)]
$\taureg(A) = \text{$\tau^-$-$\reg(A)$}$ {\rm (= (S2))};

\item[(iii)]
$\irr^\tau(A) = \irr^{\tau^-}(A)$ {\rm (= (S3))};

\item[(iv)]
$A \cong \cP(Q,\tilde{W})$ for some hyperpotential $\tilde{W}$.

\end{itemize}
\end{Thm}

It follows from results in \cite{BS25} that the analogous result holds for all algebras of finite global dimension and that in this case, the 
Jacobian algebra $\cP(Q,\tilde{W})$ in (iv) is in fact isomorphic to the path algebra $KQ$.
Here is the precise statement:

\begin{Prop}\label{prop:introhereditary1}
Suppose that $\gldim(A) < \infty$.
Then the following are equivalent:
\begin{itemize}\itemsep2mm

\item[(i)]
$E(M) = E^-(M)$ for all $M \in \md(A)$;

\item[(ii)]
$\taureg(A) = \text{$\tau^-$-$\reg(A)$}$;

\item[(iii)]
$\irr^\tau(A) = \irr^{\tau^-}(A)$;

\item[(iv)]
$A$ is hereditary.

\end{itemize}
\end{Prop}

The non-obvious parts of Proposition~\ref{prop:introhereditary1} are already proved 
in \cite{BS25}.

One can also ask when a stronger version of $E$-symmetry holds, i.e.
$E(M,N) = E(N,M)$ for all $A$-modules $M,N$.
But this can only happen in the hereditary case:

\begin{Prop}\label{prop:introhereditary2}
The following are equivalent:
\begin{itemize}\itemsep2mm

\item[(i)]
$E(M,N) = E^-(M,N)$ for
all $M,N \in \md(A)$;

\item[(ii)]
$A$ is hereditary.

\end{itemize}
\end{Prop}

For selfinjective $E$-symmetric algebras, we have quite restrictive 
homological properties:

\begin{Prop}\label{prop:introselfinjective}
Let $A$ be selfinjective and $E$-symmetric.
We also assume that $A$ is connected and not semisimple.
Then for each simple $A$-module $S$ we have 
$$
\Omega_A^4(S) \cong \nu_A^{-1}(S).
$$
In particular, $A$ is twisted $4$-periodic.
\end{Prop}

Here $$\Omega_A\df \underline{\md}(A) \to \underline{\md}(A)$$ denotes the syzygy functor, and
$$\nu_A := D\Hom_A(-,{_A}A)\df \underline{\md}(A) \to \underline{\md}(A)$$ is the Nakayama functor.
Both are autoequivalences of the stable module category $\underline{\md}(A)$.
The definition of a twisted periodic algebra is recalled in Section~\ref{sec:5}.

\subsection{Conventions}
For maps $f\df U \to V$ and $g\df V \to W$, we denote their
composition by $gf\df U \to W$.

For a linear map $f$ of vector spaces, let $\rk(f) := \dim \Ima(f)$
be its \emph{rank}.

We assume that the set $\N$ of natural numbers includes $0$.

\subsection{Organization of the article}
After the introduction (= Section~\ref{sec:1}),
in Section~\ref{sec:2} we recall
a few definitions and results from the representation theory
of finite-dimensional algebras.
Propositions~\ref{prop:introhereditary1} 
(= Proposition~\ref{prop:hereditary1}) and 
\ref{prop:introhereditary2} (= Proposition~\ref{prop:hereditary2}) 
are both proved in Section~\ref{sec:3}.
Section~\ref{sec:4} contains the proof
of our main result Theorem~\ref{thm:main1} (= Theorem~\ref{thm:monomial3}).
Finally, in Section~\ref{sec:5} we prove
Proposition~\ref{prop:introselfinjective} (= Proposition~\ref{prop:selfinjective}).


\section{Preliminaries and known results}\label{sec:2}


\subsection{Dimension vectors}
Let $\md(A)$ be the category of $A$-modules.
Let $\proj(A)$ (resp. $\inj(A)$) be the subcategory of projective (resp. injective) $A$-modules.

Let $S(1),\ldots,S(n)$ be the simple $A$-modules, up to isomorphism.
Furthermore, let $P(1),\ldots,P(n)$
be the indecomposable projective 
$A$-modules, up to isomorphism.
These are numbered such that 
$$
\tp(P(i)) \cong S(i) 
$$
for $1 \le i \le n$.
For $S = S(i)$, we also write $P(S)$ for $P(i)$.

For $M \in \md(A)$, the
\emph{Jordan-H\"older multiplicity} of $S(i)$ in $M$ is denoted by
$[M:S(i)]$.
We have $$\hom_A(P(i),M) = [M:S(i)].$$
Then
$$
\dimv(M) := ([M:S(1)],\ldots,[M:S(n)])
$$
is the
\emph{dimension vector} of $M$.

\subsection{Projective presentations of maximal rank and 
$\tau$-regular modules}\label{subsec:maximal}
For $M,N \in \md(A)$, let
$$
r(M,N) := \max\{ \rk(f) \mid f \in \Hom_A(M,N) \}.
$$
The following result characterizes all $\tau$-regular modules
in terms of their projective presentations.

\begin{Thm}[{\cite[Theorem~1.1]{BS25}}]
For $M \in \md(A)$, let
$$
P_1 \xrightarrow{f} P_0 \to M \to 0
$$
be a minimal projective presentation.
Then
the following are equivalent:
\begin{itemize}\itemsep2mm

\item[(i)]
$M$ is $\tau$-regular;

\item[(ii)]
$\rk(f) = r(P_1,P_0)$.

\end{itemize}
\end{Thm}

\subsection{Reduction technique}

In this section, 
let $B = A/AeA$ for some idempotent $e = e^2$ in $A$.
Then we interpret
$\md(B)$ as a subcategory of $\md(A)$ in the usual way.

\begin{Prop}\label{prop:reduction1}
For $M \in \md(B)$, we have
$$
E_A(M) = E_B(M).
$$
\end{Prop}

\begin{proof}
This is contained in the proof of
\cite[Theorem~7.2]{BS25}.
\end{proof}

\begin{Cor}[{\cite[Theorem~1.2]{BS25}}]\label{cor:reduction2}
We have
$$
\taureg(B) = \taureg(A) \cap \md(B).
$$
\end{Cor}

\begin{Cor}\label{cor:reduction3}
If $A$ is $E$-symmetric, then $B$ is $E$-symmetric.
\end{Cor}

\begin{proof}
Use Proposition~\ref{prop:reduction1} and its dual.
\end{proof}

\begin{Cor}\label{cor:reduction4}
If $A$ satisfies {\rm (S3)}, then $B$ satisfies {\rm (S3)}.
\end{Cor}

\begin{proof}
This follows from \cite[Corollary~7.4]{BS25}.
\end{proof}


\section{$E$-symmetric algebras of finite global dimension}
\label{sec:3}


The following lemma can be found in
\cite[Section~IV.4]{ARS97}.
There is also a dual statement.

\begin{Lem}\label{lem:pdim1stable}
For $M,N \in \md(A)$, the following hold:
If $\pdim(N) \le 1$, then
$$
\overline{\Hom}_A(M,\tau_A(N)) = \Hom_A(M,\tau_A(N)).
$$
\end{Lem}

Recall that $A$ is \emph{hereditary} if
each submodule of
a projective $A$-module is again projective.
This is equivalent to $$
\gldim(A) \le 1.
$$

\begin{Prop}\label{prop:hereditary2}
The following are equivalent:
\begin{itemize}\itemsep2mm

\item[(i)]
$E(M,N) = E^-(M,N)$ for
all $M,N \in \md(A)$;

\item[(ii)]
$A$ is hereditary.

\end{itemize}
\end{Prop}

\begin{proof}
(i) $\implies$ (ii):
Let $P \in \proj(A)$, and let $U$ be an indecomposable submodule
of $P$.
Let $M := P$ and $N := \tau_A(U)$.
We get
\begin{align*}
E(M,N) &= \hom_A(N,\tau_A(M)) = 0,
\\
E^-(M,N) &= \hom_A(\tau_A^-(N),P) = 
\hom_A(\tau_A^-(\tau_A(U)),P). 
\end{align*}
Assuming (i) and keeping in mind that
$\hom_A(U,P) \not= 0$, it follows that $U$ must be projective.
This implies that $A$ is hereditary.

(ii) $\implies$ (i):
If $A$ is hereditary, then the Auslander-Reiten formula in
Theorem~\ref{thm:ARformula} combined with 
Lemma~\ref{lem:pdim1stable} and its dual
implies that (i) holds.
\end{proof}

\begin{Prop}\label{prop:hereditary1}
Suppose that $\gldim(A) < \infty$.
Then the following are equivalent:
\begin{itemize}\itemsep2mm

\item[(i)]
$E(M) = E^-(M)$ for all $M \in \md(A)$;

\item[(ii)]
$\taureg(A) = \text{$\tau^-$-$\reg(A)$}$;

\item[(iii)]
$\irr^\tau(A) = \irr^{\tau^-}(A)$;

\item[(iv)]
$A$ is hereditary.

\end{itemize}
\end{Prop}

\begin{proof}
We know already that (i) $\implies$ (ii) $\implies$ (iii)
for all algebras $A$.
The implication (iii) $\implies$ (iv) is part of
\cite[Corollary~1.8]{BS25} and relies on the assumption $\gldim(A) < \infty$.
Finally, (iv) $\implies$ (i) is a special case of
Proposition~\ref{prop:hereditary2}.
\end{proof}

Up to Morita equivalence, hereditary algebras are the path algebras $KQ$ of acyclic quivers $Q$.
Trivially, these are then Jacobian algebras $\cP(Q,W)$, where $W = 0$
is the
 zero potential.


\section{Monomial $E$-symmetric algebras }\label{sec:4}


\subsection{Paths and cycles in quivers}
Recall that a \emph{quiver} is a quadruple $Q = (Q_0,Q_1,s,t)$, where $Q_0$ and $Q_1$ are the finite sets of \emph{vertices} and \emph{arrows}, respectively.
For $a \in Q_1$, the starting and terminal vertex of $a$ are denoted by $s(a)$ and $t(a)$, respectively.

A \emph{path} of length $m \ge 1$ in $Q$ is a sequence
$p = a_m \cdots a_2a_1$ with $a_i \in Q_1$ for $1 \le i \le m$ and
$s(a_{i+1}) = t(a_i)$ for $1 \le i \le m-1$.
Let $|p| := m$ denote the \emph{length} of $p$.
For $1 \le i \le m$, we say that $a_i$ \emph{belongs to} $p$.
For $1 \le i \le j \le m$, we call $a_j \cdots a_{i+1}a_i$ a \emph{subpath} of $p$.
Define $s(p) := s(a_1)$ and $t(p) := t(a_m)$.
Furthermore, let 
$$
\{ p\}_0 := \{ s(a_i) \mid 1 \le i \le m \} \cup \{ t(a_m) \}.
$$
For $k \in \{ p \}_0$, we say that $p$ \emph{runs through} $k$.

A path $p$ of length $|p| = 1$ such that $s(p) = t(p)$ is a \emph{loop}.

Additionally, for each vertex 
$k \in Q_0$, there is a path $e_k$ of length $0$ with $s(e_k) = t(e_k) = k$ .

A \emph{cycle} of length $m \ge 1$ in $Q$ is a path 
$C = a_m \cdots a_2a_1$ of length $m$ such that $s(C) = t(C)$.
For $k \in Q_0$ and such a cycle $C$, let
$$
{\rm mult}_k(C) := |\{ 1 \le i \le m \mid s(a_i) = k \}|
$$
be the \emph{multiplicity} of $k$ in $C$.
We call $C$ \emph{multiplicity-free} if ${\rm mult}_k(C) \le 1$ for
all $k \in Q_1$.
For $r \ge 1$, let $C^r := C \cdots CC$ be the composition of
$r$ copies of $C$.

For a cycle $C = a_m \cdots a_2a_1$ of length $m \ge 1$, we write
$C \sim_r a_{i-1} \cdots a_1a_m \cdots a_i$ for 
$2 \le i \le m$.
We also write $C \sim_r C$.
This defines an equivalence relation on the set of all cycles in $Q$.

Let $C = a_m \ldots a_2a_1$ be a cycle of length $m \ge 1$.
A \emph{cycle decomposition} of $C$ is a tuple
$(D_t,\ldots,D_1)$ such that the following hold:
\begin{itemize}\itemsep2mm

\item[(i)]
Each $D_i$ is a cycle of length $m_i \ge 1$;

\item[(ii)]
$C \sim_r D_t \cdots D_2D_1$;

\item[(iii)]
${\rm mult}_{s(D_i)}(D_i) = 1$ for each $1 \le i \le t$.

\end{itemize}
Note that $s(D_i) = s(D_j)$ for all $1 \le i,j \le t$.

\subsection{Monomial algebras and relation cycles}
For the rest of Section~\ref{sec:4},
let $A = KQ/I$ be a \emph{monomial algebra}.
Thus $Q$ is a quiver, and the ideal $I$ is generated by a set of paths of length at least $2$ such that $KQ/I$ is finite-dimensional.
Let $R$ be a minimal set of such paths such that $R$ generates $I$.
This \emph{minimal relation set} $R$ is uniquely determined by $I$.
We call the paths in $R$ \emph{minimal relations} for $A$.

Let $C = a_m \cdots a_2a_1$ be a cycle of length $m \ge 1$.
Then this is a \emph{relation cycle} for $A$ of length $m$ and
\emph{relation length} $r$ provided the following hold:
\begin{itemize}\itemsep2mm

\item[(i)]
The arrows $a_1,\ldots,a_m$ are pairwise different;

\item[(ii)]
Each subpath of length $r$ of $C^r$ is in the minimal relation set 
$R$.

\end{itemize} 
The relations in (ii) are called the \emph{relations associated to} 
the relation cycle $C$.

\subsection{Idempotent ideals}
As before, let $A = KQ/I$ be a monomial algebra.
For a subset $U \subseteq Q_0$, define
$$
e_U := \sum_{k \in U} e_k.
$$
(If $U = \varnothing$, then $e_U := 0$.)
Then we obviously have $e_U^2 = e_U$, and
the ideal $Ae_UA$ is generated by $\{ e_k \mid k \in U \}$.
Furthermore, for any idempotent $e=e^2$ in $A$, there is a subset
$U \subseteq Q_0$ such that $AeA = Ae_UA$.

For each vertex $k \in Q_0$,  we have $A/Ae_kA \cong KQ'/I'$ where
$Q'$ is obtained from $Q$ by deleting the vertex $k$ and all arrows incident to $k$, and $I'$ is generated by the minimal set of relations
$$
R' := \{ \rho \in R \mid k \notin \{ \rho \}_0 \}. 
$$

\subsection{Characterization of monomial 
$1$-Iwanaga-Gorenstein algebras}
The following is a reformulation of a theorem by Lu and Zhu.

\begin{Thm}[{Lu--Zhu \cite[Theorem~5.1.4]{LZ21}}]\label{thm:monomial1}
Let $A = KQ/I$ be a monomial algebra, and let $R$ be the associated minimal
relation set.
Then the following are equivalent:
\begin{itemize}\itemsep2mm

\item[(i)]
$A$ is $1$-Iwanaga-Gorenstein;

\item[(ii)]
There is a disjoint union 
$$
Q_1 = Q_1' \cup \bigcup_{1 \le i \le q} Q_1^{(i)}
$$
such that for each $1 \le i \le q$
the arrows in $Q_1^{(i)}$ form a relation cycle $C_i$ for $A$ of length $|C_i| = |Q_1^{(i)}|$ and relation
length 
$$
r_i \ge 2,
$$ 
and
$$
R = \bigcup_{1 \le i \le q} \{ \text{relations associated to $C_i$}  \}. 
$$

\end{itemize}
\end{Thm}

\begin{Cor}\label{cor:loops}
Let $KQ/I$ be a monomial algebra which is $1$-Iwanaga-Gorenstein.
Then the following hold:
\begin{itemize}\itemsep2mm

\item[(i)]
Each loop $h \in Q_1$ forms a relation cycle $C = h$ for $A$;

\item[(ii)] 
For each $k \in Q_0$, there is at most one loop $h \in Q_1$ 
with $s(h) = k$.

\end{itemize}
\end{Cor}

\begin{proof}
(i): Since $\dim(A) < \infty$, we have $h^r = 0$ for some $r \ge 2$.
Now (i) follows from Theorem~\ref{thm:monomial1}(ii).

(ii):
Suppose there are loops $h_1 \not= h_2$ in $Q_1$ with
$s(h_1) = s(h_2)$.
Then Theorem~\ref{thm:monomial1}(ii) together with (i) implies that
$(h_2h_1)^l \notin I$ for all $l \ge 1$.
A contradiction, since $\dim(A) < \infty$.
\end{proof}

\Example:
Let $A = KQ/I$, where $Q$ is the quiver
$$
\xymatrix{
1 \ar@<0.5ex>[r]^{a_1} & 
2 \ar@<0.5ex>[r]^{a_2}\ar@<0.5ex>[l]^{a_4} & 
3 \ar@<0.5ex>[l]^{a_3} \ar@<0.5ex>[r]^{b_1} & 4 \ar@<0.5ex>[r]^{b_2} & 5 \ar@/^2ex/[ll]^{b_3}
}
$$
and $I$ is generated by 
$$
R = \{ a_2a_1,\; a_3a_2,\; a_4a_3,\; 
a_1a_4,\; b_3b_2b_1,\; b_1b_3b_2,\; b_2b_1b_3 \}.
$$
Then $A$ is $1$-Iwanaga-Gorenstein. 
There is a relation cycle $C_1 = a_4a_3a_2a_1$ 
of length $4$ and relation length $2$, 
and a relation cycle $C_2 = b_3b_2b_1$ of length $3$ and
relation length $3$.
The minimal relation set $R$ consists of the relations associated to $C_1$ and $C_2$.
Note that $A/Ae_1A$ and $A/Ae_3A$ are not $1$-Iwanaga-Gorenstein.
The relation cycle $C_1$ has $4$ cycle decompositions, namely 
$$
(a_4a_3a_2a_1),\quad (a_1a_4,a_3a_2),\quad (a_3a_2,a_1a_4),\quad (a_2a_1a_4a_3).
$$
Here is an illustration of $C_1$:
$$
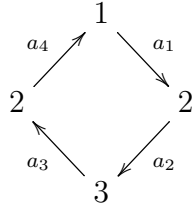
\xymatrix@-1ex{
& 1 \ar[dr]^{a_1}
\\
2 \ar[ur]^{a_4} && 2 \ar[dl]^{a_2}
\\
&3 \ar[ul]^{a_3}
}
$$
The other relation cycle $C_2$ has $3$ cycle decompositions.

\begin{Thm}\label{thm:monomial2}
Let $A = KQ/I$ be a monomial algebra, and let $R$ be the associated minimal
relation set.
Then the following are equivalent:
\begin{itemize}\itemsep2mm

\item[(i)]
$A/AeA$ is $1$-Iwanaga-Gorenstein for each $e = e^2$ in $A$;

\item[(ii)]
There is a disjoint union 
$$
Q_1 = Q_1' \cup \bigcup_{1 \le i \le q} Q_1^{(i)}
$$
such that for each $1 \le i \le q$
the arrows in $Q_1^{(i)}$ form a relation cycle $C_i$ for $A$ of length $|C_i| = |Q_1^{(i)}|$ and relation
length 
$$
r_i \ge |C_i| - 1,
$$ 
and
$$
R = \bigcup_{1 \le i \le q} \{ \text{relations associated to $C_i$}  \}. 
$$

\end{itemize}
\end{Thm}

\begin{proof}
(ii) $\implies$ (i):
Assume that (ii) holds.
First, observe that (ii) implies that $A$ is $1$-Iwanaga-Gorenstein
by
Theorem~\ref{thm:monomial1}.
Now
pick a vertex $k$ of $Q$.
If $k \in \{ C_i \}_0$ for some relation cycle $C_i$ for $A$ of relation length $r_i$, then for each
relation $\rho$ associated to $C_i$ we have 
$k \in \{ \rho \}_0$.
(Here we use that $r_i \ge |C_i|-1$.)
Thus none of these relations is a path in the quiver of $A/Ae_kA$.
All relation cycles $C_j$ for $A$ with $k \notin \{C_j\}_0$ remain intact.
Thus their relations 
are the minimal relations for the monomial algebra $A/Ae_kA$.

Therefore $A/Ae_kA$ still satisfies condition (ii).
Now one can proceed by induction and gets that $A/AeA$ is $1$-Iwanaga-Gorenstein for all $e=e^2$ in $A$.

(i) $\implies$ (ii):
Assume that (i) holds.
Since $A$ is $1$-Iwanaga-Gorenstein, we get relation cycles for $A$ as described in Theorem~\ref{thm:monomial1}(ii).

Let $C = a_m \cdots a_2a_1$ be a relation cycle for $A$ of length $m \ge 1$ 
and relation length $r \ge 2$.
We have to show that $r \ge |C|-1$.

If there is a loop 
$h$ which belongs to $C$, then $C = h$, compare Corollary~\ref{cor:loops}(i).
In this case, we are done.

Thus, we can assume that $m \ge 2$ and that there are no loops belonging to  $C$.
Let $(D_t,\ldots,D_1)$ be a cycle decomposition of $C$.
Without loss of generality, we assume that 
$C = D_t \cdots D_2D_1$, $\min\{|D_t|,\ldots,|D_1|\} = |D_1|$ and that $D_1$ is multiplicity-free.
We can also assume that $D_1 = a_s \cdots a_2a_1$ and $s(a_i) = i$ for $1 \le i \le s$.
Thus $D_1$ looks as in Figure~\ref{fig:1}(1).

Claim 1: For each $1 \le i \le t$, we have $r \ge |D_i| - 1$.

Proof: Assume that $r \le |D_i|-2$ for some $i$.
Then there is a minimal relation $\rho \in R$ associated to
$C$ such that $1 \notin \{ \rho \}_0$.
Thus $\rho$ is still a minimal relation for $A/Ae_1A$.
However, since we removed $e_1$, we see that $\rho$ is no 
longer associated to a relation cycle for $A/Ae_1A$.
Thus by Theorem~\ref{thm:monomial1}, $A/Ae_1A$ is not $1$-Iwanaga-Gorenstein, a contradiction. $\square$

Claim 2: Let $t \ge 2$. Then we have $r \le |D_1|$. 

Proof: Let $t \ge 2$ and $r > |D_1|$. Then 
$(a_s \cdots a_2a_1)^l \notin I$ for all $l \ge 1$. A contradiction,
since $\dim(A) < \infty$. $\square$

Thus for $t \ge 2$, we have $r= s$ or $r = s-1$.

Claim 3: 
Let $t \ge 2$. 
Then for $2 \le i \le t$, we have
$\{ D_i \}_0 \subseteq \{ D_1 \}_0$.

Proof:
Since $|D_1| = s$ and $r \le s$, we know that
there is a subpath $\rho$ of $D_1$ such that
$\rho$ is a minimal relation for $A$.
Assume now that $k \in \{ D_i \}_0 \setminus \{ D_1 \}_0$.
Then $\rho$ does not run through $k$. 
Thus $\rho$ is still a minimal relation for $A/Ae_kA$.
However, 
$\rho$ is no 
longer associated to a relation cycle for $A/Ae_kA$.
Thus $A/Ae_kA$ is not $1$-Iwanaga-Gorenstein, a contradiction.
$\square$

Claim 4: 
We have $t=1$.

Proof:
Assume that $t \ge 2$, and let $D_i = b_{s_i} \cdots b_2b_1$. 
We know that there is no loop belonging to $C$.
Thus $|D_i| \ge |D_1| \ge 2$. 
We obviously have $s(a_1) = s(b_1)$.
We know from Claim 3 that $t(b_1) \in \{ 1,\ldots,s \}$.
Since there are no loops in $C$, we have $t(b_1) \not= 1$.
Assume that $t(b_1) \ge 3$. 
Thus the path $p := a_s \cdots a_{t(b_1)+1}a_{t(b_1)}$ has length at most $|D_1|-2$.
We know that $r \ge |D_i|-1$.
It follows that $(pb_1)^l \notin I$ for all $l \ge 1$.
A contradiction, since $\dim(A) < \infty$.
Thus we get $t(b_1) = t(a_1)$ and therefore $s(b_2) = s(a_2)$.
Analogously one shows now that $t(b_i) = t(a_i)$
for $2 \le i \le s$.
Thus we are in the situation of Figure~\ref{fig:1}(2).
We get
$(a_sb_{s-1} \cdots a_2b_1)^l \notin I$ for all $l \ge 1$,
a contradiction.
$\square$

We have shown that $C = D_1$
Combining this with Claim 1 yields that for each relation cycle
$C$ of relation length $r$, we have $r \ge |C|-1$.
This finishes the proof.
\end{proof}

\begin{Cor}\label{cor:monomial2}
Each relation cycle appearing in Theorem~\ref{thm:monomial2} is multiplicity-free.
\end{Cor}

\begin{proof}
This is a consequence of the final paragraph of the proof of Theorem~\ref{thm:monomial2}(ii).
\end{proof}

\begin{figure}[h]
\renewcommand{\arraystretch}{4.0}
\begin{tabular}{lll}
$
\xymatrix@-3.5ex{
\text{(1)}&& 1 \ar[drr]^{a_1}&&
\\
s \ar[urr]^{a_s}&&&& 2 \ar[dd]^{a_2}
\\
&
\\
\cdots \ar[uu]&&&& 3 \ar[dll]^{a_3}
\\
&& 4 \ar[ull]&&
}
$
&
$
\xymatrix@-3.5ex{
\text{(2)} && 1 \ar[drr]^{a_1}\ar@<-1ex>[drr]_{b_1}&&
\\
s \ar[urr]^{a_s}\ar@<-1ex>[urr]_{b_s}&&&& 2 \ar[dd]^{a_2}\ar@<-1ex>[dd]_{b_2}
\\
&
\\
\cdots \ar[uu]\ar@<-1ex>[uu]&&&& 3 \ar[dll]^{a_3}\ar@<-1ex>[dll]_{b_3}
\\
&& 4 \ar[ull]\ar@<-1ex>[ull]
}
$
\end{tabular}
\caption{Illustrations for the proof of Theorem~\ref{thm:monomial2}}.
\label{fig:1}
\end{figure}
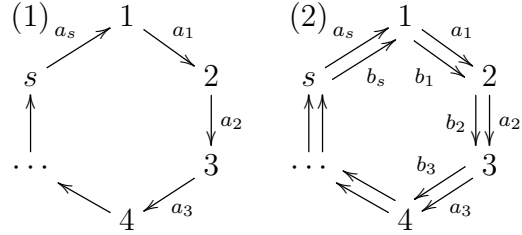

\subsection{Characterization of monomial $E$-symmetric  algebras}

\begin{Thm}\label{thm:monomial3}
Let $A = KQ/I$ be a monomial algebra.
Then the following are equivalent:
\begin{itemize}\itemsep2mm

\item[(i)]
$E(M) = E^-(M)$ for all $M \in \md(A)$ {\rm (= (S1))};

\item[(ii)]
$\taureg(A) = \tau^-$-$\reg(A)$ {\rm (= (S2))};

\item[(iii)]
$\irr^\tau(A) = \irr^{\tau^-}(A)$ {\rm (= (S3))};

\item[(iv)]
$A \cong \cP(Q,\tilde{W})$ for some hyperpotential $\tilde{W}$.

\end{itemize}
\end{Thm}

\begin{proof}
We know already from the introduction that
$$
\text{(i) $\implies$ (ii) $\implies$ (iii)}.
$$
By Theorem~\ref{thm:jacobian} we know also that (iv) $\implies$ (i).
We now assume that (iii) holds. We need to show that this implies (iv).

By \cite[Theorem~1.7]{BS25}, condition (iii) implies that $A/AeA$ is $1$-Iwanaga-Goren\-stein for each $e=e^2$ in $A$.
Thus we are in the situation of Theorem~\ref{thm:monomial2}(ii).

If all relation cycles $C_i$ have relation length $r_i = t_i|C_i| -1$
for some $t_i \ge 1$, then the associated relations form a
reduced hyperpotential $\tilde{W}$ for $Q$ and we get 
$A = KQ/I \cong \cP(Q,\tilde{W})$.

Let $C = a_s \cdots a_2a_1$ be a relation cycle for $A$ of relation length
$r$.
We know from Theorem~\ref{thm:monomial2}(ii) 
and Corollary~\ref{cor:monomial2} that 
$C$ is multiplicity-free and that 
$$
r \ge s-1.
$$
We illustrate $C$ in Figure~\ref{fig:1}(1).
We assume now without loss of generality that 
$s(a_1),\ldots,s(a_s)$ are the only vertices of $Q$, and that $s(a_i) = i$
for $1 \le i \le s$.
Here we use Corollary~\ref{cor:reduction4}.

If $r = s-1$, then we are done, since the relations associated to $C$ will be part of a hyperpotential $\tilde{W}$.

Thus, we assume from now on that 
$$
r \ge s.
$$

Claim 1: There is no $b \in Q_1$ with
$s(b) \not= t(b)$ and
$b \not= a_i$ for all $1 \le i \le s$.

Proof:
Assume there is such an arrow $b$. Without loss of generality, let
$s(b) = 1$ and $t(b) = i \not= 1$.
This is illustrated in Figure~\ref{fig:2}(1).
Then $(a_s \cdots a_{i+1}a_ib)^l \notin I$ for all $l \ge 1$. 
A contradiction, since $\dim(A) < \infty$. $\square$

Claim 2:
$Q$ has at most one loop.

Proof:
By Corollary~\ref{cor:loops}(ii), there is at most one loop at each vertex of $Q$.
Next, suppose that $Q$ has two different loops $h_1$ and $h_2$.
Without loss of generality let $s(h_1) = 1$ and $s(h_2) = i \not= 1$.
We illustrated this in Figure~\ref{fig:2}(2).
Then $
(a_s \cdots a_ih_ia_{i-1} \cdots a_1h_1)^l \notin I$ for all $l \ge 1$. 
A contradiction, since $\dim(A) < \infty$. $\square$

Claim 3: $Q$ has no loop.

Proof: Assume that $Q$ has exactly one loop $h$.
Without loss of generality, let $s(h) = 1$.
This situation is illustrated in Figure~\ref{fig:2}(3).
If $r > s$, then $(a_s \cdots a_2a_1h)^l \notin I$ for all $l \ge 1$.
A contradiction, since $\dim(A) < \infty$. 
Thus $r=s$.
There is a minimal relation $h^p \in R$ for some $p \ge 2$.
Let $M$ be the uniserial string module
$$
M = M(a_{s-1} \cdots a_2a_1h^{p-1}).
$$
Note that $\dimv(M) = (p,1,\ldots,1)$.
Then we have $E(M) = 0$ and $E^-(M) = p-1 \not= 0$.
This is easily calculated by computing the minimal projective presentation 
$$
P(2)^{p-1} \to P(1) \to M \to 0
$$
and the minimal injective presentation 
$$
0 \to M \to I(s) \to I(s)^{p-1}
$$
of $M$.
This gives us the $g$-vectors 
$g(M) = (-1,p-1,0\ldots,0)$ and $g^-(M) = (0,\ldots,0,p-2)$, respectively.
Then one uses the well known formulas
\begin{align*}
E(M) &= \dim \End_A(M) + \bil{g(M),\dimv(M)},
\\
E^-(M) &= \dim \End_A(M) + \bil{g^-(M),\dimv(M)}.
\end{align*}
Thus the irreducible component
$$
\cZ := \overline{\cO_M}
$$
in generically $\tau$-regular but not $\tau^-$-regular.
This proves that $\irr^\tau(A) \not= \irr^{\tau^-}(A)$,
a contradiction. $\square$

Claim 4: 
We have $r = ts-1$ for some $t \ge 2$.

Proof:
We know by Claims 1 and 3 that the only arrows of $Q$ are $a_1,\ldots,a_s$. 
Thus $A$ is a cyclic Nakayama algebra.
Then the claim follows from
\cite[Proposition~1.10]{BS25}. $\square$

This finishes the proof.
\end{proof}

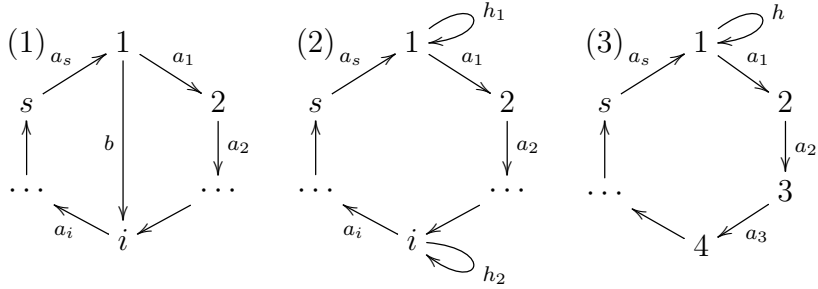
\begin{figure}[h]
\renewcommand{\arraystretch}{4.0}
\begin{tabular}{lll}
$
\xymatrix@-3.5ex{
\text{(1)} && 1 \ar[drr]^{a_1}\ar[dddd]_b&&
\\
s \ar[urr]^{a_s}&&&& 2 \ar[dd]^{a_2}
\\
&&&&
\\
\cdots \ar[uu]&&&& \cdots \ar[dll]
\\
&& i \ar[ull]^{a_i}&&
}
$
&
$
\xymatrix@-3.5ex{
\text{(2)}&& 1 \ar[drr]^{a_1}\ar@(ur,r)[]^{h_1}&&
\\
s \ar[urr]^{a_s}&&&& 2 \ar[dd]^{a_2}
\\
&
\\
\cdots \ar[uu]&&&& \cdots \ar[dll]
\\
&& i \ar[ull]^{a_i}\ar@(r,dr)[]^{h_2}&&
}
$
&
$
\xymatrix@-3.5ex{
\text{(3)}&& 1 \ar[drr]^{a_1}\ar@(ur,r)[]^{h}&&
\\
s \ar[urr]^{a_s}&&&& 2 \ar[dd]^{a_2}
\\
&
\\
\cdots \ar[uu]&&&& 3 \ar[dll]^{a_3}
\\
&& 4 \ar[ull]&&
}
$
\end{tabular}
\caption{Illustrations for the proof of Theorem~\ref{thm:monomial3}}.
\label{fig:2}
\end{figure}

\subsection{Related results}
The following theorem is another characterization of
monomial Jacobian algebras.
For the missing definitions, we refer to the cited publications.

\begin{Thm}[\cite{KR07}, \cite{A09}, \cite{GE21}]\label{thm:monomial4}
Let $A = KQ/I$ be a monomial algebra.
Then the following are equivalent:
\begin{itemize}\itemsep2mm

\item[(i)]
$A$ is $2$-Calabi-Yau tilted;

\item[(ii)]
$A$ is $1$-Iwanaga-Gorenstein and stably $3$-Calabi-Yau;

\item[(iii)]
$A \cong \cP(Q,\tilde{W})$ for some hyperpotential $\tilde{W}$.

\end{itemize}
\end{Thm}

The implication (ii) $\implies$ (iii) is due to Garc{\'i}a Elsener \cite[Theorem~3.5 and Remark~3.6]{GE21}, and (i) $\implies$ (ii)
and (iii) $\implies$ (i) are general results (which hold for all finite-dimensional $K$-algebras $A$) by
Keller--Reiten \cite[Theorem~3.3]{KR07} and
Amiot \cite[Corollary~3.6]{A09}, respectively.
Ladkani \cite{L14} proved that one can work with hyperpotentials intead of potentials.

In general, 
for a $1$-Iwanaga-Gorenstein algebra $A$, it is unclear how the 
additional assumptions
stably $3$-Calabi-Yau and $E$-symmetric are related.


\section{Selfinjective $E$-symmetric algebras}\label{sec:5}


\subsection{Twisted periodic algebras}

Let $M \in \md(A)$.
Then $M$ is
\emph{$\Omega$-periodic}
if 
$$
\Omega_A^m(M) \cong M
$$
for some $m \ge 1$.
Here $\Omega_A(M)$ is by definition the kernel of the projective
cover $P \to M$.

Let $A^e :=  A \otimes_K A^\op$ denote the
\emph{enveloping algebra} of $A$.
Recall that $A^e$ acts on $A$ by
$$
(x \otimes y)a := xay.
$$
Then
$A$ is a \emph{periodic algebra}
if $A$ is $\Omega$-periodic as an $A^e$-module.
In this case, all indecomposable non-projective $M \in \md(A)$ are $\Omega$-periodic, see for example \cite[Theorem~11.19(i)]{SY11}.

For a $K$-algebra automorphism $\sigma\df A \to A$, let
${_\sigma}A_1$ be the $A^e$-module defined by
$$
(x \otimes y)a := \sigma(x)ay
$$
where the underlying vector space is $A$.
Then
$A$ is 
\emph{twisted $m$-periodic} (or just \emph{twisted periodic})
with respect to $\sigma$
if
there exists some $m \ge 1$ such that 
$$
\Omega^m_{A^e}(A) \cong {_\sigma}A_1
$$ 
in $\md(A^e)$.

Obviously, each periodic algebra is twisted periodic.
Green--Snashall--Solberg \cite[Lemma~1.5]{GSS03} showed  that
twisted periodic algebras are selfinjective.

\begin{Thm}[{Green, Snashall, Solberg \cite[Theorem~1.4]{GSS03}}]\label{thm:GSS}
Assume that $A$ is connected and not semisimple.
For $m \ge 1$, 
the following are equivalent:
\begin{itemize}\itemsep2mm

\item[(i)]
$\Omega_A^m(A/J(A)) \cong A/J(A)$;

\item[(ii)]
$A$ is twisted $m$-periodic.

\end{itemize}
\end{Thm}

\subsection{Selfinjective $E$-symmetric algebras are twisted 
periodic}
For the proof of Proposition~\ref{prop:selfinjective}, we need the following easy and well known lemma.

\begin{Lem}\label{lem:stablehom}
Let $A$ be selfinjective, and let $S$ be a simple and non-projective 
$A$-module.
Then for each $M \in \md(A)$, we have
$$
\Hom_A(S,\tau_A(M)) = \underline{\Hom}_A(S,\tau_A(M)).
$$
\end{Lem}

\begin{Prop}\label{prop:selfinjective}
Let $A$ be selfinjective and $E$-symmetric.
We also assume that $A$ is connected and not semisimple.
Then the following hold:
\begin{itemize}\itemsep2mm

\item[(i)]
For each simple $A$-module $S$, we have $\tau_A^2(S) \cong \nu_A(S)$;

\item[(ii)]
$A$ is twisted $4$-periodic.

\end{itemize}
\end{Prop}

\begin{proof}
(i):
Assume that $S$ and $T$ are simple $A$-modules.
Then they are non-projective, since $A$ is connected and not semisimple. In particular $\tau_A^m(S) \not= 0$ for all $m \in \Z$.

Let $P := P(T)$ and $M := P \oplus \tau_A(S)$.
Note that $\nu_A^{-1}(T) \cong \soc(P)$.
Then
$$
\tau_A(M) \cong \tau_A^2(S) 
\text{\quad and \quad}
\tau_A^-(M) \cong S.
$$
We get
\begin{align*} E(M) =
\hom_A(M,\tau_A(M)) &= [\tau_A^2(S):T] + 
\hom_A(\tau_A(S),\tau_A^2(S))
\\
&\ge  [\tau_A^2(S):T] + \dim \overline{\Hom}_A(\tau_A(S),\tau_A^2(S))
\\
&= [\tau_A^2(S):T] + \dim \underline{\Hom}_A(S,\tau_A(S))
\\
&= [\tau_A^2(S):T] + \hom_A(S,\tau_A(S)).
\end{align*}
(For the last equality, we used Lemma~\ref{lem:stablehom}.)
Furthermore, we have
$$
E^-(M) = \hom_A(\tau_A^-(M),M) = \hom_A(S,P) + \hom_A(S,\tau_A(S)). 
$$
Note that
$$
\hom_A(S,P) =
\begin{cases}
1 & \text{if $S \cong \nu_A^{-1}(T)$},
\\
0 & \text{otherwise}.
\end{cases}.
$$
Since $E(M) = E^-(M)$,
we get 
$$
\hom_A(S,P) \ge [\tau_A^2(S):T].
$$
Therefore
$$
[\tau_A^2(S):T] =
\begin{cases}
1 & \text{if $S \cong \nu_A^{-1}(T)$},
\\
0 & \text{otherwise}.
\end{cases}.
$$
(We used here that $\tau_A^2(S) \not= 0$.)
Thus $\tau_A^2(S) \cong \nu_A(S)$.

(ii):
Since $A$ is selfinjective,
we have functorial isomorphisms
$$
\tau_A \cong \Omega_A^2 \circ \nu_A \cong \nu_A \circ \Omega_A^2\df
\underline{\md}(A) \to \underline{\md}(A).
$$
Thus the condition $\tau_A^2(S) \cong \nu_A(S)$ is equivalent to $\Omega_A^4(S) \cong \nu_A^{-1}(S)$.
Now Theorem~\ref{thm:GSS} says that $A$ is twisted $4$-periodic.
\end{proof}

There are several related results for $2$-Calabi-Yau tilted algebras
and selfinjective Jacobian algebras.
We refer to
Valdivieso-D{\'i}az \cite[Proposition~2.1]{V16},
Garc{\'i}a Elsener--Schiffler \cite[Corollary~3.6]{GES17} and
Jasso--Muro \cite[Proposition~2.2.19]{JM27} for details.

\bigskip

\subsection*{Statement on AI usage}
We used the AI model ChatGPT Pro to find the module $M$ at the end of the proof of Theorem~\ref{thm:monomial3}.
(With one hour of trying, we would have found it by ourselves...)
This is the only place in the paper where AI was used.

\subsection*{Acknowledgements}
JS was partially funded by the Deutsche Forschungsgemeinschaft
(DFG, German Research Foundation) under Germany's Excellence Strategy - GZ 2047/1, Projekt-ID 390685813.
He also thanks the
Faculty of Mathematics and Computer Science
of the Nicolaus Copernicus University in Toru\'n for 10 days of hospitality in September 2026.


\end{document}